\documentclass[12pt]{amsart}
\usepackage[margin=1in]{geometry}
\usepackage{amsmath,amssymb,amsthm,mathtools}
\usepackage{tikz}
\usepackage{placeins}
\usetikzlibrary{patterns}
\usepackage[hidelinks]{hyperref}
\numberwithin{equation}{section}
\newtheorem{theorem}{Theorem}[section]
\newtheorem{lemma}[theorem]{Lemma}
\newtheorem{proposition}[theorem]{Proposition}
\theoremstyle{definition}
\newtheorem{remark}[theorem]{Remark}
\theoremstyle{plain}
\newcommand{\Z}{\mathbb Z}
\newcommand{\R}{\mathbb R}
\newcommand{\C}{\mathbb C}
\newcommand{\T}{\mathbb T}
\newcommand{\spec}{\operatorname{spec}}
\newcommand{\Dom}{\operatorname{Dom}}
\newcommand{\Rea}{\operatorname{Re}}
\newcommand{\Ima}{\operatorname{Im}}
\title[Irreducibility of the Bloch variety]{Irreducibility of the Bloch variety for periodic Schr\"odinger operators in arbitrary dimension}
\author[W.\ Liu]{Wencai Liu}
\address{[W.\ Liu] Department of Mathematics, Texas A\&M University}
\email{\href{mailto:wencail@tamu.edu}{wencail@tamu.edu}}
\begin{document}
\maketitle

\begin{abstract}
Let $d\ge1$ and $V\in C(\T^d;\C)$. We prove that the Bloch
variety of $-\Delta+V$ is irreducible modulo periodicity.
The proof combines holomorphic continuation of spectral band
functions with integer translations of quasimomentum. Using analytic
geometry and perturbation estimates along imaginary rays, we
show that, after suitable integer translations, all irreducible
components contain the same holomorphic spectral branch.
\end{abstract}

\section{Introduction and main results}

We study the periodic Schr\"odinger operator $-\Delta+V$.
Let $d\ge1$, $\T^d=\R^d/(2\pi\Z)^d$, and $V\in C(\T^d;\C)$.

The Bloch variety $B_V$ consists of all $(k,\lambda)\in\C^d\times\C$
for which the eigenvalue equation
\[
 (-\Delta+V)u=\lambda u
\]
with the boundary condition
\begin{equation}\label{g4}
 u(x+2\pi n)=e^{2\pi i k\cdot n}u(x)
 \quad(n\in\Z^d)
\end{equation}
has a nonzero solution $u\in H^2_{\rm loc}(\R^d)$. 
The parameter $k\in\C^d$ is called the (complex) quasimomentum.
Writing $u(x)=e^{ik\cdot x}v(x)$ with $v$ periodic, we obtain
\[
 H_V(k)=(-i\nabla+k)^2+V,
\]
with $\Dom H_V(k)=H^2(\T^d)$ for all $k\in\C^d$.
Thus
\[
 B_V=\{(k,\lambda)\in\C^d\times\C:\lambda\in\spec H_V(k)\}.
\]
By \eqref{g4}, $B_V$ is invariant under integer translations in $k$:
if $(k,\lambda)\in B_V$, then $(k+m,\lambda)\in B_V$ for all $m\in\Z^d$.
It is a classical result that  $B_V$ is a complex analytic
subset of $\C^{d+1}$ (see, e.g.,  \cite[Theorem 5.14]{Kuchment2016}).

For example, when $V=0$,
\[
 B_0=\bigcup_{n\in\Z^d}\Gamma_n,
\]
where
\[
 \Gamma_n=\{(k,\lambda)\in\C^d\times\C:\lambda=\sum_{j=1}^d(k_j+n_j)^2\}.
\]

A longstanding conjecture states that $B_V$ is irreducible modulo
periodicity (see, e.g., \cite[Conjecture 5.17]{Kuchment2016}).
Here, $B_V$ is \emph{irreducible modulo periodicity} if, for any two
irreducible components $\Gamma_1$ and $\Gamma_2$ of $B_V$, there exists
$m\in\Z^d$ such that
\[
 \Gamma_1=\Gamma_2+(m,0).
\]
The one-dimensional case follows from classical Floquet theory
\cite[Theorem 1.19]{Kuchment2016}. In dimension two, Kn\"orrer
and Trubowitz \cite{KT} proved irreducibility by constructing a
directional compactification of the complex Bloch variety.
We refer to \cite{Kuchment2016} for further history.
In this paper, we prove the irreducibility  result in every dimension.
\begin{theorem}\label{thm:main}
For every $d\ge1$ and every $V\in C(\T^d;\C)$, the Bloch variety $B_V$ is
irreducible modulo periodicity.
\end{theorem}

For discrete periodic Schr\"odinger operators, the problem becomes
algebraic after the Floquet transform. Early irreducibility results were
obtained by B\"attig \cite{Battig1988,Battig1992} and Gieseker,
Kn\"orrer, and Trubowitz \cite{GKT} in dimensions two and three, by
compactification methods. In \cite{Liu}, the author developed a method
to prove irreducibility of a family of Laurent polynomials, and used it
to prove irreducibility results for Bloch and Fermi varieties in
arbitrary dimensions. Further results for periodic graphs were obtained in
\cite{FaustLiu,FaustLopezGarcia,FLM,FLM2}.
\subsection{Idea of the proof}
We write $z^2=\sum_{j=1}^d z_j^2$ for $z\in\C^d$ and assume that
$V$ has zero average. In the Fourier basis (see Section~\ref{sec:basics}),
\begin{equation*}
 H_0(k)=\operatorname{diag}_n d_n(k),
\end{equation*}
where $d_n(k)=(n+k)^2$, and
\begin{equation*}
 H_V(k)=\operatorname{diag}_n d_n(k)+V.
\end{equation*}

For a suitably chosen fixed $b\in\R^d$, bounded $a\in\R^d$,
and sufficiently large $s$, $d_0(a+isb)$ is well separated from
the other   diagonal entries $d_n(a+isb)$, $n\ne0$.
Analytic perturbation theory (see Lemma \ref{lem:simple}) then gives a simple holomorphic eigenvalue
branch $\Lambda(k)$ close to $d_0(k)=k^2$
(see Figure~\ref{fig:intro-ray}).

Fix $p_0=(is_0b,\Lambda(is_0b))$ for large $s_0$.
By simplicity, this point belongs to a unique irreducible component.
Our goal is to show that every irreducible component $\Gamma$ of $B_V$
contains $p_0$ after an integer translation. Let
\[
 Y=\bigcup_{m\in\Z^d}\bigl(\Gamma+(m,0)\bigr).
\]
By Theorem~\ref{thm:projection}, proved using Ueda's lemma, there exists
$\lambda_{s_0}$ such that $(is_0b,\lambda_{s_0})\in Y$.

We need not have $\lambda_{s_0}=\Lambda(is_0b)$:
$\lambda_{s_0}$ may be close to $d_{n_{s_0}}(is_0b)=(n_{s_0}+is_0b)^2$
for some $n_{s_0}\ne0$.
To connect $(is_0b,\lambda_{s_0})\in Y$ to $p_0$, we combine
continuation of spectral branches with integer shifts of quasimomentum.
This shift argument, inspired by the author's proof of geometric
Borg's theorem \cite[Remark 3]{LiuBorg}, plays a crucial role in our proof.

The difficulty is that $\lambda_{s_0}$ and $\Lambda(is_0b)$ may be
close to different bands, and crossings of spectral bands
may obstruct holomorphic continuation (see the red part of Figure \ref{fig:intro-ray}). Instead, we  first extend
  the eigenvalue continuously along the ray $k=isb$,
writing $\lambda_s$ for  the extension, with $(isb,\lambda_s)\in Y$ (Lemma~\ref{lem:bounded-indices}). See the red+blue parts of Figure \ref{fig:intro-ray}.

Along this continuation, $\lambda_s$ is close to $d_{n_s}(isb)$,
where the index $n_s$ may change with $s$.
Our key observation is that $\lambda_s+s^2|b|^2$ stays in one fixed
vertical strip (Remark~\ref{rem:shifted-strips}). In particular, there exists a nonnegative integer $N$ such that 
\[
 |n_s|^2=N, \text{ for all } s\geq s_0.
\]

By Lemma~\ref{lem:unique-index}, there exist $s_1>s_0$ ($s_1$ depends on $N$)
and $n_*\in\Z^d$ such that $n_s=n_*$ for all $s\ge s_1$.
Moreover, $\lambda_s$ lies on a simple holomorphic branch
(see the blue part of Figure~\ref{fig:intro-ray}), and 
for $s\ge s_1$,  
\[
 \lambda_s=\Lambda(n_*+isb).
\]
\begin{figure}[htbp]
	\centering
	\begin{tikzpicture}[x=1cm,y=1cm,>=stealth,font=\small,
	selected/.style={draw=blue,line width=1pt},
	continuous/.style={draw=red!80!black,line width=1pt},
	simple/.style={draw=black!80,line width=1pt}]
	\coordinate (start) at (1.6,3.3);
	\coordinate (p0) at (1.6,1.15);
	\coordinate (transition) at (5.3,3.8);
	\draw[->,black!70] (0.35,0) -- (12.1,0) node[right] {$s$};
	\draw[dashed,black!35] (1.6,0) -- (1.6,3.9);
	\draw[dashed,black!35] (5.3,0) -- (5.3,4.45);
	\draw (1.6,0.06) -- (1.6,-0.06) node[below=4pt] {$s_0$};
	\draw (5.3,0.06) -- (5.3,-0.06) node[below=4pt] {$s_1$};
	
	\draw[continuous] (0.65,3.68) -- (start)
	-- (1.95,3.0) -- (2.35,3.5)
	-- (2.75,3.05) -- (3.15,3.62)
	-- (3.55,3.25);
	\draw[selected] (3.55,3.25) -- (3.95,3.65)
	.. controls (4.2,3.9) and (4.8,3.8) .. (transition);
	\draw[selected,->] (transition)
	.. controls (6.5,3.8) and (7.4,4.23) .. (8.25,4.65);
	\draw[selected] (8.25,4.65)
	.. controls (9.1,5.08) and (9.9,5.65) .. (10.6,6.05);
	\fill[red!80!black] (start) circle (1.7pt);
	\node[anchor=north east] at (1.45,3.03) {$\lambda_{s_0}$};
	\node at (3.55,4.45) {continuous};
	\node[rotate=27] at (7.9,4.95) {holomorphic};
	\node[anchor=west,align=left] at (10.75,5.97)
	{$\lambda_s=\Lambda(n_*+isb)$\\[-1pt]\footnotesize close to $d_{n_*}(isb)$};
	\node at (3.45,2.45) {$|n_s|^2=|n_{s_0}|^2$};
	\node at (6.6,3.25) {$n_s=n_*$};
	
	\draw[simple] (0.65,1.17)
	.. controls (0.95,1.14) and (1.3,1.14) .. (p0)
	.. controls (4.0,1.2) and (5.8,1.8) .. (7.3,2.65)
	.. controls (8.4,3.27) and (9.05,3.75) .. (9.6,4.2);
	\node[rotate=29] at (7.1,2.02) {holomorphic};
	\fill[black!80] (p0) circle (1.7pt);
	\node[below right=2pt,fill=white,inner sep=1pt] at (p0) {$p_0$};
	\node[anchor=west,align=left] at (9.75,4.13)
	{$\Lambda(isb)$\\[-1pt]\footnotesize close to $d_0(isb)$};
	\end{tikzpicture}
	\caption{Spectral branches along $k=isb$.
		The eigenvalue $\lambda_s$ is continuous for $s\ge s_0$.
		For $s\ge s_1$, $n_s=n_*$ and $\lambda_s=\Lambda(n_*+isb)$
		is simple and holomorphic.
		The holomorphic branch $\Lambda(isb)$ passes through $p_0$.}
	\label{fig:intro-ray}
\end{figure}
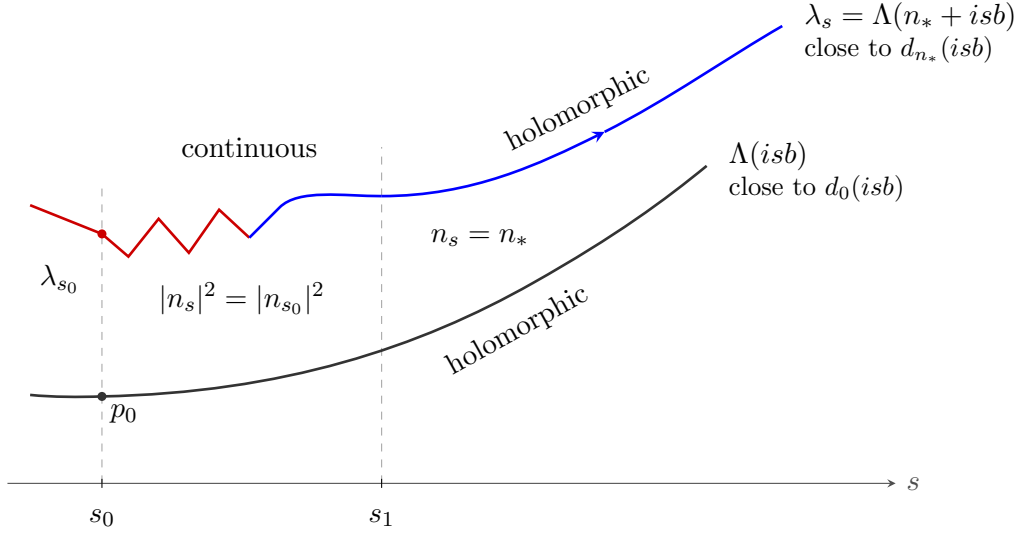
\FloatBarrier
By the integer translation invariance of $Y$, shifting the
quasimomentum from $isb$ to $n_*+isb$ leads to
\[(n_*+isb,\Lambda(n_*+isb)) \in Y.\]
We then continue $\Lambda(k)$ from $k=n_*+isb$ to $k=isb$,
and along the ray to $k=is_0b$.
Thus $p_0=(is_0b,\Lambda(is_0b))\in Y$, completing the proof.

\begin{figure}[htbp]
\centering
\begin{tikzpicture}[x=1cm,y=1cm,>=stealth,font=\small,
  every node/.style={inner sep=4pt},
  patharrow/.style={->,line width=.7pt}]
  \node (initial) at (0,2.1) {$(isb,\lambda_s)$};
  \node[above=3pt] at (initial.north) {$\lambda_s=\Lambda(n_*+isb)$};
  \node (shifted) at (7.2,2.1)
    {$(n_*+isb,\Lambda(n_*+isb))$};
  \node (principal) at (7.2,0) {$(isb,\Lambda(isb))$};
  \node (fixed) at (0,0) {$p_0$};
  \draw[patharrow] (initial.east) -- node[above] {integer shift}
    (shifted.west);
  \draw[patharrow] (shifted.south) --
    node[right,align=left] {holomorphic\\continuation\\at fixed $s$}
    (principal.north);
  \draw[patharrow] (principal.west) --
    node[above] {holomorphic continuation} node[below] {$s_0\leftarrow s$}
    (fixed.east);
\end{tikzpicture}
\caption{ Holomorphic continuation  along paths.}
\label{fig:intro-shift}
\end{figure}
\FloatBarrier

\section{Basics of analytic geometry and spectral theory}\label{sec:basics}
\subsection{Analytic geometry}
A set $A\subset\C^N$ is \emph{complex analytic} if for every
$p\in\C^N$ there is an open set $U$ containing $p$ such that
\[
 A\cap U=\{z\in U:f_1(z)=\cdots=f_r(z)=0\},
 \qquad f_j\in\mathcal O(U),
\]
where $\mathcal O(U)$ denotes the holomorphic functions on $U$.

A nonempty complex analytic set $A\subset\C^N$ is
\emph{irreducible} if  $A=Y_1\cup Y_2$ with
$Y_1,Y_2$ complex analytic subsets of $\C^N$ implies either
$Y_1=A$ or $Y_2=A$.
An \emph{irreducible component of $A$} is a maximal
irreducible complex analytic subset of $A$.

\subsection{Basics of spectral theory}
Since
\[
 B_{V-c}=\{(k,\lambda-c):(k,\lambda)\in B_V\},
\]
 it suffices to assume
\begin{equation}\label{eq:mean}
\int_{\T^d}V dx=0.
\end{equation}

For $g\in L^\infty(\T^d)$, define its Fourier coefficients by
\[
 \widehat g(r)=\int_{\T^d}g(x)e^{-ir\cdot x}\,dx,
 \qquad r\in\Z^d.
\]
We also use $g$ for the multiplication operator
$(gu)(x)=g(x)u(x)$ on $L^2(\T^d)$. Then
\begin{equation}\label{eq:multiplication-operator}
 \|g\|=\|g\|_\infty.
\end{equation}
Here $\|\cdot\|$ denotes the operator norm. 

In the Fourier basis $e_n(x)=e^{in\cdot x}$, we identify $g$ with
its matrix $g(m,n)=\widehat g(m-n)$ on $\ell^2(\Z^d)$: for any $f\in \ell^2(\Z^d)$
\begin{equation*}
(gf)(n)=\sum_{m\in \Z^d} \widehat{g}(n-m) f(m).
\end{equation*}
This convention applies
to $V$ and to $P,E$ below.

Recall that $d_n(k)=(n+k)^2$. In the Fourier/matrix  representation,
\begin{equation}\label{eq:fourier-representation0}
H_0(k)=\operatorname{diag}_n d_n(k),
\end{equation}
and 
\begin{equation}\label{eq:fourier-representation}
 H_V(k)=\operatorname{diag}_n d_n(k)+V.
\end{equation}

For $z\notin\{d_n(k):n\in\Z^d\}$, the free resolvent is
\[
R_0(k,z)=\operatorname{diag}_n(d_n(k)-z)^{-1}
\]
and
\[
\|R_0(k,z)\|=\bigl(\inf_n|d_n(k)-z|\bigr)^{-1}.
\]

Denote by  $M=\|V\|_\infty$.

If $\inf_n|d_n(k)-z|>M$, the Neumann series gives
\begin{equation}\label{eq:compact-inverse}
(H_V(k)-z)^{-1}
=R_0(k,z)[I+VR_0(k,z)]^{-1}.
\end{equation}
Also
\begin{equation}\label{eq:spectral-enclosure}
\spec(H_0(k)+tV)
\subset\bigcup_{n\in\Z^d}
\{z:|z-d_n(k)|\le tM\},\qquad 0\le t\le1.
\end{equation}

We use the following standard facts
\cite{Kato}.

\begin{proposition}\label{prop:local-polynomial}
	For every $k\in\C^d$, $H_V(k)$ has compact resolvent. Its spectrum
	consists of isolated eigenvalues of finite algebraic multiplicity
	$m_{\rm alg}(k,\lambda)$, with no finite accumulation point.
	
	Near every $(k_*,\lambda_*)\in B_V$, there are a neighborhood
	$W_\alpha$ of $k_*$, a disk $\Delta_\alpha$ centered at $\lambda_*$,
	and an $\ell_\alpha\times\ell_\alpha$ holomorphic matrix $A_\alpha(k)$,
	where $\ell_\alpha=m_{\rm alg}(k_*,\lambda_*)$, such that the polynomial
	\[
	P_\alpha(k,\lambda)=\det(\lambda I-A_\alpha(k))
	\]
	satisfies
	\[
	B_V\cap(W_\alpha\times\Delta_\alpha)=\{P_\alpha=0\}
	\]
	and
	\[
	P_\alpha(k_*,\lambda)=(\lambda-\lambda_*)^{\ell_\alpha}.
	\]

\end{proposition}
 
\begin{remark}
By discreteness of the spectrum, choose a disk $\Delta_\alpha$
centered at $\lambda_*$ with
$\overline{\Delta_\alpha}\cap\spec H_V(k_*)=\{\lambda_*\}$.
After shrinking $W_\alpha$, $\partial\Delta_\alpha\cap\spec H_V(k)
=\varnothing$ for all $k\in W_\alpha$. The Riesz projection is defined as
\begin{equation}\label{eq:riesz-projection}
 \Pi_\alpha(k)=\frac1{2\pi i}\int_{\partial\Delta_\alpha}
       (z-H_V(k))^{-1}\,dz.
\end{equation}
It has rank $\ell_\alpha$. Choose a basis $u_1,\ldots,u_{\ell_\alpha}$ of
$\operatorname{ran}\Pi_\alpha(k_*)$ and set
$U_\alpha(k)c=\sum_{j=1}^{\ell_\alpha}c_j\Pi_\alpha(k)u_j$
for $c\in\C^{\ell_\alpha}$. After shrinking $W_\alpha$,
$U_\alpha(k):\C^{\ell_\alpha}\to\operatorname{ran}\Pi_\alpha(k)$
is an isomorphism. Define
\begin{equation}\label{eq:riesz-matrix}
 A_\alpha(k)=U_\alpha(k)^{-1}H_V(k)U_\alpha(k).
\end{equation}
By \eqref{eq:riesz-projection} and \eqref{eq:riesz-matrix},
$A_\alpha$ is holomorphic, and the roots of $P_\alpha(k,\cdot)$
are precisely the eigenvalues of $H_V(k)$ in $\Delta_\alpha$,
with the same algebraic multiplicities.
\end{remark}

\begin{lemma}\cite[Chapter 1]{Simon}\label{lem:singular-values}
Let $A$ be a compact operator on a complex Hilbert space $\mathcal H$.
Let $\lambda_j(A)$ be its nonzero eigenvalues, counted with algebraic
multiplicity, and let $s_j(A)$ be its singular values, listed in
nonincreasing order.
Then, for every $q\geq 1$,
\[
 \sum_j|\lambda_j(A)|^q\le\sum_j s_j(A)^q.
\]

For every bounded operator $B$ on $\mathcal H$,
\[
 s_j(AB)\le\|B\|s_j(A).
\]
\end{lemma}

\section{Projection of irreducible components}\label{sec:projection}
In this section, we use Ueda's lemma to prove the following theorem.
\begin{theorem}\label{thm:projection}
	For $m\in\Z^d$, let $T_m(k,\lambda)=(k+m,\lambda)$.
Let $\Gamma$ be an irreducible component (nonempty) of $B_V$, and set
\[
 Y=\bigcup_{m\in\Z^d}T_m(\Gamma).
\]
Define
\[
 E_Y(k)=\{\lambda\in\C:(k,\lambda)\in Y\}
\]
and
\[
 \mathcal O_Y=\{k\in\C^d:E_Y(k)\ne\varnothing\}.
\]
Then $\mathcal O_Y=\C^d$.
\end{theorem}

\subsection{Local equations}
Fix $\Gamma$ and $Y$ as in Theorem~\ref{thm:projection}.
Since $T_m(B_V)=B_V$, each $T_m(\Gamma)$ is an irreducible
component of $B_V$. Then
distinct sets $T_m(\Gamma)$ form a locally finite family \cite{Chirka}.
Thus $Y$ is closed and analytic.
 By the definition of $Y$, 
\begin{equation}\label{eq:invariance}
(k,\lambda)\in Y
\quad\text{if and only if}\quad
(k+m,\lambda)\in Y.
\end{equation}
Fix $(k_*,\lambda_*)\in Y$ and consider a  neighborhood
$W_\alpha\times \Delta_\alpha$  as in  Proposition \ref{prop:local-polynomial}.
Let $Q_\alpha$ be the product of the distinct monic local irreducible
factors of $P_\alpha$ whose zero sets belong to $Y$. It is monic in
$\lambda$, with holomorphic coefficients in $k$
\cite[Chapter 1]{Chirka}.
Then $Q_\alpha\mid P_\alpha$. Set $q_\alpha=\deg_\lambda Q_\alpha$. 
By shrinking $W_\alpha\times \Delta_\alpha$, one has that
\begin{equation}\label{eq:selected}
Y\cap(W_\alpha\times\Delta_\alpha)=\{Q_\alpha=0\}
\end{equation}
and
\begin{equation}\label{eq:selected-root}
Q_\alpha(k_*,\lambda)=(\lambda-\lambda_*)^{q_\alpha}, q_\alpha\ge1.
\end{equation}

On overlaps of local coordinates, we have 
\begin{equation}\label{eq:local-units}
Q_\alpha=u_{\alpha\beta}Q_\beta,
\end{equation}
where $u_{\alpha\beta}$ is holomorphic and nowhere zero.

For fixed $k$, we say that an isolated zero $\lambda$ of $f(k,\cdot)$ has order
$\operatorname{ord}_{\mu=\lambda}f(k,\mu)=m$ if
\[
 \partial_\mu^j f(k,\lambda)=0\quad(0\le j<m),\text{ and }
 \partial_\mu^m f(k,\lambda)\ne0.
\]

For all $\lambda\in E_Y(k)$, set
\[
 m_Y(k,\lambda):=\operatorname{ord}_{\mu=\lambda}Q_\alpha(k,\mu).
\]
By \eqref{eq:local-units}, this definition is independent of $\alpha$.
By $Q_\alpha\mid P_\alpha$ and Proposition~\ref{prop:local-polynomial},
we obtain
\begin{equation}\label{eq:multiplicity-bound}
 1\le m_Y(k,\lambda)\le m_{\rm alg}(k,\lambda).
\end{equation}

\subsection{Ueda's lemma}
\begin{lemma}[Ueda {\cite[Lemma 3, p.~260]{Ueda}}]\label{lem:ueda}
Let $G\subset\C^d$ be open, and let $A\subset G\times\C$.
Suppose that an open cover $\{U_\alpha\}$ of $G\times\C$
by connected sets and $f_\alpha\in\mathcal O(U_\alpha)$,
$f_\alpha\not\equiv0$, satisfy
\[
 A\cap U_\alpha=\{(k,\lambda)\in U_\alpha:f_\alpha(k,\lambda)=0\}
\]
and
\[
 f_\alpha=u_{\alpha\beta}f_\beta,
\]
where $u_{\alpha\beta}$ is holomorphic and nowhere zero on
$U_\alpha\cap U_\beta$.
Assume that $E(k)=\{\lambda:(k,\lambda)\in A\}$ is discrete for every $k\in G$.
For $\lambda\in E(k)$, set
\[
 m(k,\lambda)=\operatorname{ord}_{\mu=\lambda}f_\alpha(k,\mu),
 \qquad (k,\lambda)\in U_\alpha.
\]
This is finite by discreteness and independent of $\alpha$.
Let $\lambda_{\mathrm{ref}}\in\C$ satisfy
$\lambda_{\mathrm{ref}}\notin E(k)$ for all $k\in G$.
Suppose that, for some integer $p\ge1$ and every compact $K\subset G$,
\begin{equation}\label{eq:ueda-condition}
 \lim_{r\to\infty}\sup_{k\in K}
 \sum_{\substack{\lambda\in E(k)\\|\lambda-\lambda_{\mathrm{ref}}|>r}}
 m(k,\lambda)|\lambda-\lambda_{\mathrm{ref}}|^{-p}=0.
\end{equation}
Then
\begin{equation}\label{eq:ueda}
 T_p(k)=\sum_{\lambda\in E(k)}m(k,\lambda)(\lambda-\lambda_{\mathrm{ref}})^{-p}
\end{equation}
converges absolutely and is holomorphic on $G$.
\end{lemma}
 
\subsection{Proof of Theorem~\ref{thm:projection}}
\begin{proof}
Since $Y\ne\varnothing$, choose $(k_*,\lambda_*)\in Y$.
Fix $R>|k_*|$ and set $p=\lfloor d/2\rfloor+2$. Let
\[
 G=\{k\in\C^d:|k|<R\}.
\]
It suffices to prove $G\subset\mathcal O_Y$ for every  $R$.
By \eqref{eq:fourier-representation0} and \eqref{eq:spectral-enclosure}, for all $k\in G$ and
$\lambda\in\spec H_V(k)$,
\begin{equation}\label{eq:positive-weight}
 \Rea(\lambda+\tau)^{-p}>0,
\end{equation}
where $\tau$ is sufficiently large.

In particular, $-\tau\notin\spec H_V(k)$ for all $k\in G$.
Consider
\[
 T(k)=\sum_{\lambda\in E_Y(k)}m_Y(k,\lambda)(\lambda+\tau)^{-p}.
\]
By \eqref{eq:positive-weight},
\begin{equation}\label{g9}
T(k)=0\quad\text{if and only if}\quad E_Y(k)=\varnothing.
\end{equation}
The  eigenvalues of $(H_V(k)+\tau)^{-1}$ are
$(\lambda+\tau)^{-1}$, with the same algebraic multiplicities.
Since $p-1>d/2$, by \eqref{eq:fourier-representation0}, \eqref{eq:compact-inverse},
and Lemma~\ref{lem:singular-values}, we obtain
\begin{equation}\label{eq:schatten-sum}
\begin{aligned}
 K&:=\sup_{k\in G}\sum_{\lambda\in\spec H_V(k)}
 m_{\rm alg}(k,\lambda)|\lambda+\tau|^{1-p}\\
 &\le C\sup_{k\in G}\sum_{n\in\Z^d}|d_n(k)+\tau|^{1-p}<\infty.
\end{aligned}
\end{equation}
By \eqref{eq:multiplicity-bound} and \eqref{eq:schatten-sum},
\begin{equation}\label{eq:selected-count}
 \sup_{k\in G}
 \sum_{\substack{\lambda\in E_Y(k)\\|\lambda+\tau|>r}}
 m_Y(k,\lambda)|\lambda+\tau|^{-p}
 \le K/r\longrightarrow0.
\end{equation}
By  \eqref{eq:selected-count} and applying Lemma~\ref{lem:ueda}   with
$\lambda_{\mathrm{ref}}=-\tau$,  and local equations $f_\alpha=Q_\alpha$ near $Y$
and $f_\alpha=1$ off $Y$, 
we have that $T\in\mathcal O(G)$.
By \eqref{eq:positive-weight} and \eqref{g9}, 
\begin{equation}\label{g8}
 \Rea T\ge0,\text{ and }
\Rea T(k_*)>0.
\end{equation}
Since $\Rea T$ is harmonic,  by \eqref{g8},  we conclude that $\Rea T>0$ on $G$.
Therefore, by \eqref{g9},  $G\subset\mathcal O_Y$. Since $R>|k_*|$ is arbitrary,
$\mathcal O_Y=\C^d$.
\end{proof}
\section{Proof of Theorem~\ref{thm:main}}\label{sec:proof}

Fix $b\in\R^d$ such that for all $r\in\Z^d\backslash\{0\}$,
\[
 b\cdot r\ne0.
\]
Set $\rho=M+1$ (recalling that $M=\|V\|_\infty$) and
\[
 \mathcal U=\{k\in\C^d:|d_r(k)-d_0(k)|>2\rho
                         \text{ for all }r\in\Z^d\backslash\{0\}\}.
\]

\begin{lemma}\label{lem:simple}
The set $\mathcal U$ is nonempty and open. For every $k\in\mathcal U$,
the disk $\{\lambda: |\lambda-k^2|<\rho\}$ contains exactly one eigenvalue,
counted algebraically. Denote this simple eigenvalue by $\Lambda(k)$.
Then $\Lambda(k)$ is holomorphic on $\mathcal U$.
If $Y$ is as in Theorem~\ref{thm:projection},
$k:[0,1]\to\mathcal U$ is continuous, and
$(k(0),\Lambda(k(0)))\in Y$, then for all $0\le t\le1$, 
\begin{equation}\label{eq:simple-persistence}
 (k(t),\Lambda(k(t)))\in Y.
\end{equation}
For every $R\ge0$, there exists $s_1=s_1(V, R,b)\ge1$ such that for every $s\geq s_1$,
\begin{equation}\label{eq:uniform-separation}
 \{a+isb:a\in\R^d,\ |a|\le R\}\subset\mathcal U.
\end{equation}
\end{lemma}
\begin{proof}
For $|\Rea k|\le R$, we have
\begin{equation}\label{eq:large-mode-separation}
 \Rea(d_r(k)-d_0(k))
 =|r|^2+2r\cdot\Rea k
 \ge |r|^2-2R|r|\longrightarrow\infty
 \quad\text{as }|r|\to\infty.
\end{equation}
By \eqref{eq:large-mode-separation}, $\mathcal U$ is locally defined
by finitely many  strict inequalities and hence is open.

For all $a\in\R^d$ and $r\in\Z^d\setminus\{0\}$,
\begin{equation}\label{eq:ray-mode-separation}
 |d_r(a+isb)-d_0(a+isb)|
 \ge2s|b\cdot r|\longrightarrow\infty
 \quad\text{as }s\to\infty.
\end{equation}
By \eqref{eq:large-mode-separation} and \eqref{eq:ray-mode-separation},
we obtain \eqref{eq:uniform-separation}, and hence
$\mathcal U\ne\varnothing$.

Fix $k\in\mathcal U$. Then $|d_n(k)-d_0(k)|>2\rho$ for all $n\ne0$.
For $\lambda$ with $|\lambda-k^2|\le\rho$ and $n\ne0$, we obtain
\begin{equation}\label{eq:isolated-contour}
\begin{aligned}
 |\lambda-d_n(k)|
 &\ge |d_n(k)-d_0(k)|-|\lambda-d_0(k)|\\
 &>2\rho-\rho=\rho>M.
\end{aligned}
\end{equation}
By a perturbation argument, we have that  $\Lambda (k)$ is the only eigenvalue in the disk $\{\lambda: |\lambda-k^2|<\rho\}$, and hence $\Lambda (k)\in\mathcal O(\mathcal U)$.

Near $(k,\Lambda(k))$, $B_V$ is the graph of $\Lambda(k)$
and belongs to a unique irreducible component. 
This proves
\eqref{eq:simple-persistence}.
\end{proof}

\begin{lemma}\label{lem:fourier}
There exists $s_0(V,b)\ge1$ such that for every $s\geq s_0$, 
\begin{equation}\label{eq:enclosure}
 \spec H_V(isb)\subset
 \bigcup_{n\in\Z^d}\{\lambda\in\C:|\lambda-(n+isb)^2|\le\tfrac14\}.
\end{equation}
\end{lemma}
\begin{proof}
By \eqref{eq:mean} and the Stone--Weierstrass theorem
\cite[Theorem 7.33]{Rudin}, there exists a finite
set $F\subset\Z^d\setminus\{0\}$ and a
trigonometric polynomial
\[
 P(x)=\sum_{r\in F}p_re^{ir\cdot x}
\]
such that the remainder $E=V-P$ satisfies
\begin{equation}\label{eq:approximation}
 \|E\|_\infty<\tfrac1{16}.
\end{equation}

In Fourier coordinates,
\begin{equation}\label{eq:polynomial-operator}
 (Pu)(m)=\sum_{r\in F}p_r u(m-r).
\end{equation}
Suppose that $\lambda\in\spec H_V(isb)$ and
$|\lambda-d_n(isb)|>1/4$ for every $n\in\Z^d$. Set
\[
 D_s=\operatorname{diag}_n(d_n(isb)-\lambda)^{-1}.
\]
Then
\begin{equation}\label{eq:diagonal-resolvent-bound}
 \|D_s\|\le4.
\end{equation}
For $r\in F$ and $m\in\Z^d$,
\begin{equation}\label{eq:shift-separation}
 |d_{m+r}(isb)-d_m(isb)|\ge2s|b\cdot r|.
\end{equation}
By \eqref{eq:polynomial-operator}, \eqref{eq:diagonal-resolvent-bound},
and \eqref{eq:shift-separation}, we obtain
\begin{equation}\label{eq:double-resolvent}
\begin{aligned}
 \|D_sPD_s\|
 &\le\sum_{r\in F}|p_r|\sup_m
 \frac1{|d_{m+r}(isb)-\lambda|\,|d_m(isb)-\lambda|}\\
 &\le\frac4s\sum_{r\in F}\frac{|p_r|}{|b\cdot r|}.
\end{aligned}
\end{equation}
By Proposition~\ref{prop:local-polynomial}, choose $u\in H^2$ with
$H_V(isb)u=\lambda u$ and $\|u\|_2=1$. Since
\[
\begin{aligned}
 u&=-D_sVu=-D_sPu-D_sEu\\
  &=D_sPD_sVu-D_sEu,
\end{aligned}
\]
by \eqref{eq:multiplication-operator}, \eqref{eq:approximation},
\eqref{eq:diagonal-resolvent-bound}, and \eqref{eq:double-resolvent},
we obtain
\[
 1\le\frac{4M}{s}\sum_{r\in F}\frac{|p_r|}{|b\cdot r|}
             +4\|E\|_\infty<\frac12
\]
for all sufficiently large $s$. This contradiction proves
\eqref{eq:enclosure}.
\end{proof}

\begin{remark}\label{rem:shifted-strips}
Since $(n+isb)^2+s^2|b|^2=|n|^2+2is\,b\cdot n$, by \eqref{eq:enclosure}, we obtain that the eigenvalues of $H_V(isb)$, shifted by $s^2|b|^2$, satisfy 
\[
 \spec H_V(isb)+s^2|b|^2
 \subset\bigcup_{N=0}^{\infty}
 \{z\in\C:|\Rea z-N|\le\tfrac14\}
\]
for all $s\ge s_0$. Thus the lines $\Rea z=N+\tfrac12$,
$N\ge0$, contain no (shifted) eigenvalues; see
Figure~\ref{fig:shifted-strips}.
If $\lambda_s\in\spec H_V(isb)$ depends continuously on $s$
in an interval $I\subset[s_0,\infty)$, then
$\lambda_s+s^2|b|^2$ remains in one fixed strip.
\end{remark}

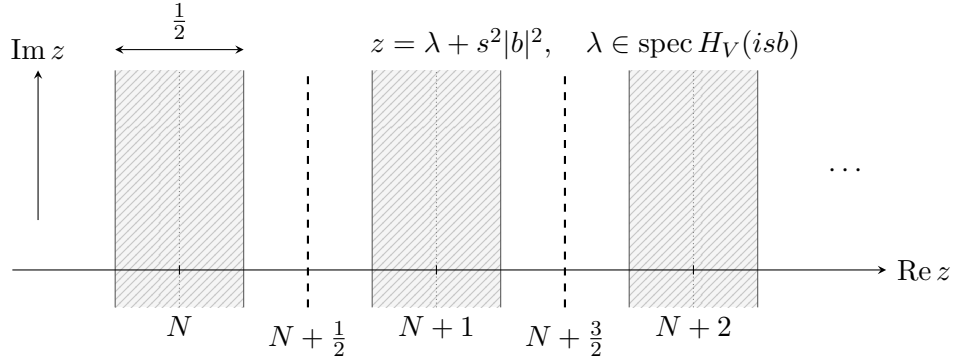
\begin{figure}[ht]
\centering
\begin{tikzpicture}[x=3.4cm,y=1.1cm,>=stealth,font=\small]
  \foreach \a in {0,1,2} {
    \fill[black!4] (\a-0.25,-0.45) rectangle (\a+0.25,2.4);
    \fill[pattern=north east lines,pattern color=black!22]
      (\a-0.25,-0.45) rectangle (\a+0.25,2.4);
    \draw[black!65] (\a-0.25,-0.45) -- (\a-0.25,2.4);
    \draw[black!65] (\a+0.25,-0.45) -- (\a+0.25,2.4);
    \draw[black!50,densely dotted] (\a,-0.45) -- (\a,2.4);
  }
  \foreach \a in {0.5,1.5} {
    \draw[dashed,thick] (\a,-0.45) -- (\a,2.4);
  }
  \draw[->] (-0.65,0) -- (2.75,0) node[right] {$\Rea z$};
  \draw[->] (-0.55,0.6) -- (-0.55,2.4) node[above] {$\Ima z$};
  \foreach \a/\lab in {0/N,1/{N+1},2/{N+2}} {
    \draw (\a,0.055) -- (\a,-0.055);
    \node[below,fill=white,inner sep=2pt] at (\a,-0.48) {$\lab$};
  }
  \node[below] at (0.5,-0.48) {$N+\tfrac12$};
  \node[below] at (1.5,-0.48) {$N+\tfrac32$};
  \draw[<->] (-0.25,2.65) -- (0.25,2.65)
    node[midway,above] {$\tfrac12$};
  \node at (2.6,1.2) {$\cdots$};
  \node[anchor=west] at (0.7,2.72)
    {$z=\lambda+s^2|b|^2,\quad \lambda\in\spec H_V(isb)$};
\end{tikzpicture}
\caption{The (shifted) eigenvalues lie in the shaded strips.
The dashed lines contain no (shifted) eigenvalues.}
\label{fig:shifted-strips}
\end{figure}

\begin{lemma}\label{lem:bounded-indices}
Let $s_0(V,b)$ be such that \eqref{eq:enclosure} holds for all $s\ge s_0$,
and let $Y$ be as in Theorem~\ref{thm:projection}.
There exists an integer $N=N(V,Y,b,s_0)\ge0$ such that, for every
$s\ge s_0$, there exist $\lambda_s\in E_Y(isb)$, continuously depending on $s$,  and $n_s\in\Z^d$
satisfying
\begin{equation}\label{eq:capture-index}
 |n_s|^2=N
\end{equation}
and
\begin{equation}\label{eq:capture}
 |\lambda_s-(n_s+isb)^2|\le\tfrac14.
\end{equation}
\end{lemma}
\begin{proof}
By Theorem~\ref{thm:projection} and \eqref{eq:enclosure}, choose
$\lambda_{s_0}\in E_Y(is_0b)$ and a corresponding $n_{s_0}\in\Z^d$ such that \eqref{eq:capture} holds for $s=s_0$.
Set $N=|n_{s_0}|^2$.
By \eqref{eq:selected}, $Q_\alpha(isb,\cdot)$ has a continuous local
root $\lambda_s\in E_Y(isb)$ through $\lambda_{s_0}$, for
$s\in I=[s_0,s_0+\delta]$ with $\delta>0$.
For every $s\in I$, by \eqref{eq:enclosure}, choose
$n_s\in\Z^d$ such that
\begin{equation}\label{eq:capture1}
 |\lambda_s-(n_s+isb)^2|\le\tfrac14.
\end{equation}

By \eqref{eq:capture1}, we obtain
\begin{equation}\label{eq:capture-real}
 \bigl|\Rea(\lambda_s+s^2|b|^2)-|n_s|^2\bigr|\le\tfrac14.
\end{equation}
Since $|n_s|^2$ are nonnegative integers, by continuity of $\lambda_s$ and
\eqref{eq:capture-real} (see Remark \ref{rem:shifted-strips}), we obtain that \eqref{eq:capture-index} holds for all $s\in I$.

By \eqref{eq:capture1} and \eqref{eq:capture-index},
\begin{equation}\label{eq:capture-bounded}
 |\lambda_s|\le N+s^2|b|^2+\tfrac14.
\end{equation}
This implies that if \eqref{eq:capture-index} and \eqref{eq:capture} hold for all $s\in [s_0,\kappa)$,  then by passing to the limit and using the closedness of $Y$, \eqref{eq:capture-index} and \eqref{eq:capture} hold for $s=\kappa.$ By continuation, we have that  \eqref{eq:capture-index} and \eqref{eq:capture} hold for  all $s\geq s_0$.

\end{proof}
The following lemma shows that when $s$ is sufficiently large, $n_s$ in  \eqref{eq:capture-index} and \eqref{eq:capture}  is independent of $s$.
\begin{lemma}\label{lem:unique-index}
Let $N$, $\lambda_s$, and $n_s$ be as in Lemma~\ref{lem:bounded-indices}.
There exist  $s_1=s_1(V,b,N,s_0)\ge s_0$ and $n_* \in \Z^d$ such that for every
$s\ge s_1$,
\begin{equation}\label{g10}
 |\lambda_s-(n_{*}+isb)^2|\le\tfrac14.
\end{equation}
Moreover, 
\begin{equation}\label{eq:holomorphic-tail}
 \lambda_s=\Lambda(n_*+isb), s\geq s_1.
\end{equation}
In particular, $s\mapsto\lambda_s$ extends holomorphically to a
complex neighborhood of $[s_1,\infty)$.
\end{lemma}
\begin{proof}
By \eqref{eq:uniform-separation} with $R=2\sqrt N$, choose $s_1\ge s_0$
such that $n+isb\in\mathcal U$ for all $s\ge s_1$ and
$n\in\Z^d$ with $|n|^2=N$.
For these $s,n$ and every $m\in\Z^d\setminus\{n\}$ with $|m|=\sqrt N$,  we obtain
\begin{equation}\label{eq:tail-separation}
 |d_m(isb)-d_n(isb)|
 =|d_{m-n}(n+isb)-d_0(n+isb)|>2\rho.
\end{equation}
By \eqref{eq:capture},  \eqref{eq:tail-separation} and
continuity of $\lambda_s$,  we conclude that $n_s$ is  constant on $[s_1,\infty)$. Denote  this constant by $n_*$. This implies \eqref{g10}.
By \eqref{eq:invariance}, \eqref{eq:capture}, and
Lemma~\ref{lem:simple}, we obtain \eqref{eq:holomorphic-tail}.

\end{proof}

\begin{proof}[\bf Proof of Theorem~\ref{thm:main}]
By Lemmas~\ref{lem:simple} and \ref{lem:fourier}, choose $s_0 (V,b)\ge1$
such that \eqref{eq:enclosure} holds for all $s\ge s_0$ and
\begin{equation}\label{eq:ray-isolated}
 isb\in\mathcal U\qquad(s\ge s_0).
\end{equation}
Fix
\[
 p_0=(is_0b,\Lambda(is_0b)),
\]
and by Lemma~\ref{lem:simple} again, there exists a  unique irreducible component    $\Gamma_0$  containing $p_0$.

For an arbitrary irreducible component $\Gamma$, recall that
$Y=\bigcup_{m\in\Z^d}T_m(\Gamma)$.
Note that the point $p_0$ is fixed independently of $\Gamma$. Therefore, in order to prove Theorem \ref{thm:main}, 
 it suffices to prove $p_0\in Y$:
then $p_0\in T_m(\Gamma)$ for some $m\in\Z^d$, and uniqueness gives
$T_m(\Gamma)=\Gamma_0$.

By Lemma \ref{lem:unique-index} and \eqref{eq:invariance}, one has that for some $n_*\in\Z^d$,
\[
 (isb ,\lambda_s=\Lambda(n_*+isb) )\in Y
\]
and
\[
 (n_*+isb,\Lambda(n_*+isb))\in Y.
\]
For $0\le t\le1$, define
\[
 k_1(t)=(1-t)n_*+isb
\]
and
\[
 k_2(t)=i\bigl((1-t)s+ts_0\bigr)b.
\]
From the proof of Lemma \ref{lem:unique-index},  when $s\geq s_1$,
both paths
lie in $\mathcal U$ . Applying \eqref{eq:simple-persistence} along
$k_1$ and then $k_2$, we obtain
\[
 (isb,\Lambda(isb))\in Y
\]
and then
\[
 p_0=(is_0b,\Lambda(is_0b))\in Y.
\]
This completes the proof.
\end{proof}
\section*{Acknowledgements}
W. Liu would like to thank the Simons Institute for the Theory of Computing
for its hospitality during his visit, when this work was finished.
This work was supported in part by NSF grant DMS-2246031.
He also thanks Peter Kuchment for sharing a 20-page AI-generated
draft obtained by Antoine Levitt, which proposes a proof of
irreducibility of the Bloch variety in dimension three.
The proposed approach  is different from ours and is based on the
directional compactification of Kn\"orrer and Trubowitz \cite{KT}.
\section*{AI usage}
The main ideas of this work came from discussions between the
author and GPT-6. After the model generated a proposed proof,
the author substantially simplified the arguments and rewrote
the manuscript (again with GPT assistance).

\bibliographystyle{amsplain}
\bibliography{bib}
\end{document}